\documentclass[12pt,a4paper,reqno,twoside]{amsart}

\usepackage[english]{babel}
\usepackage{stmaryrd}
\usepackage{dsfont}
\usepackage[symbol*,ragged]{footmisc}
\usepackage[colorlinks,linkcolor=red,anchorcolor=blue,citecolor=blue,urlcolor=blue]{hyperref}
\usepackage{color,xcolor}

\usepackage{geometry}
\usepackage{amssymb}
\usepackage{amsmath}
\usepackage{mathrsfs}
\usepackage{amsfonts}
\usepackage{epsfig}

\usepackage{amsthm}
\usepackage{amsxtra}
\usepackage{bbding}
\usepackage{epsfig}
\usepackage{graphicx}
\usepackage{latexsym}
\usepackage{mathbbol}
\usepackage{bbold}

\usepackage{pifont}
\usepackage{wasysym}
\usepackage{skull}
\usepackage{float}

\DeclareSymbolFontAlphabet{\mathbb}{AMSb}
\DeclareSymbolFontAlphabet{\mathbbol}{bbold}

\usepackage{amscd}
\usepackage[all]{xy}
\allowdisplaybreaks[4]
\usepackage{setspace}

\theoremstyle{plain}
\newtheorem{theorem}{\normalfont\scshape Theorem}[section]
\newtheorem{proposition}{\normalfont\scshape Proposition}[section]
\newtheorem{lemma}[proposition]{\normalfont\scshape Lemma}

\newtheorem*{corollary*}{\normalfont\scshape Corollary}

\theoremstyle{remark}
\newtheorem*{remark*}{\normalfont\scshape Remark}
\newtheorem*{notation}{\normalfont\scshape Notation}

\numberwithin{equation}{section}
\addtocounter{footnote}{1}

\renewcommand{\footnoterule}{
  \kern -3pt
  \hrule width 2.5in height 0.4pt
  \kern 3pt
}

\makeatletter
\@ifundefined{MakeUppercase}{}{}
\makeatother

\begin{document}
	
\title[Waring--Goldbach problem for one square and seventeen fifth powers of primes]
	  {On Waring--Goldbach problem for one square and seventeen fifth powers of primes}

\author[Min Zhang, Jinjiang Li, Fei Xue]{Min Zhang \quad \& \quad Jinjiang Li \quad \& \quad Fei Xue}

\address{School of Applied Science, Beijing Information Science and Technology University,
		 Beijing 100192, People's Republic of China}

\email{min.zhang.math@gmail.com}

\address{(Corresponding author) Department of Mathematics, China University of Mining and Technology,
         Beijing, 100083, People's Republic of China  }

\email{jinjiang.li.math@gmail.com}

\address{Department of Mathematics, China University of Mining and Technology,
         Beijing, 100083, People's Republic of China  }

\email{fei.xue.math@gmail.com}

\date{}

\footnotetext[1]{Jinjiang Li is the corresponding author. 

\quad\,\,
{\textbf{MR(2020) Subject Classification}}: 11P05, 11P32

}

\begin{abstract}
In this paper, it is established that every sufficiently large positive integer $n$ subject to $n\equiv0\pmod2$ can be represented as a sum of one square of prime and seventeen fifth powers of primes, which gives an enhancement upon the previous result of Br\"{u}dern and Kawada \cite{Brudern-Kawada-2011}.
\end{abstract}

\maketitle

\section{Introduction and main result}
It is very likely that, for each $s>1$, every sufficiently large integer can be represented
as the sum of one square and $s$ $k$--th powers of positive integers with $k\geqslant3$. To be specific, we shall be concerned with the Diophantine equation
\begin{equation}\label{ini-pro}
N=x^2+y_1^k+y_2^k+\cdots+y_s^k,
\end{equation}
where $s$ and $k$ are natural numbers, $k\geqslant 3$. This family of equations belongs to the small stock of variants of Waring's problem that have been studied by various writers since the early days of the Hardy--Littlewood method. A purely heuristical application of that method, based on a major arc analysis only, suggests that the number $R_{k,s}(n)$ of solutions to (\ref{ini-pro}) in natural numbers $x,y_1,\dots,y_s$ satisfies the asymptotic relation
\begin{equation}\label{asymp-dio}
R_{k,s}(n)=\frac{\Gamma(\frac{3}{2})\Gamma^s(1+\frac{1}{k})}{\Gamma(\frac{1}{2}+\frac{s}{k})}\mathfrak{S}_{k,s}(N)
N^{\frac{s}{k}-\frac{1}{2}}(1+o(1))
\end{equation}
as $n$ tends to infinity, provided only that $s>\frac{1}{2}k$. Here the singular series is defined by
\begin{equation*}
\mathfrak{S}_{k,s}(N)=\sum_{q=1}^\infty\frac{1}{q^{s+1}}\sum_{\substack{a=1\\ (a,q)=1}}^q
\Bigg(\sum_{x=1}^qe\bigg(\frac{ax^2}{q}\bigg)\Bigg)  \Bigg(\sum_{y=1}^qe\bigg(\frac{ay^k}{q}\bigg)\Bigg)^s
e\bigg(-\frac{aN}{q}\bigg).
\end{equation*}
A first analysis of the problem was made by Stanley \cite{Stanley-1930} in 1930. Following the pattern laid down by Hardy and Littlewood \cite{Hardy-Littlewood-1920,Hardy-Littlewood-1925} in their classic series Partitio Numerorum, Stanley \cite{Stanley-1930} established the asymptotic formula (\ref{asymp-dio}) for $s\geqslant s_1(k)$, where
\begin{equation}\label{Stanley-bound}
s_1(3)=7,\quad s_1(4)=14,\quad s_1(5)=28,\quad s_1(k)=2^{k-2}\bigg(\frac{1}{2}k-1\bigg)+O(k)\qquad (k>5).
\end{equation}

Later, Sinnadurai \cite{Sinnadurai-1965} verified (\ref{asymp-dio}) for $R_{3,6}(N)$, and Hooley \cite{Hooley-1981} gave a different proof for this result. When $k\geqslant4$, however, the authors are not aware of any improvements of Stanley's bounds (\ref{Stanley-bound}) recorded in the literature. Yet, since the 1920s, the theory of Waring's problem has experienced waves of innovation, resulting in significantly smaller lower bounds for $s$ for which (\ref{asymp-dio}) can be demonstrated. The current state of the art is that when $k\geqslant3$ and
$s\geqslant2^{k-1}+2$, then for any $\varepsilon>0$, one has
\begin{equation}\label{asymp-dio-error-modify}
R_{k,s}(n)=\frac{\Gamma(\frac{3}{2})\Gamma^s(1+\frac{1}{k})}{\Gamma(\frac{1}{2}+\frac{s}{k})}\mathfrak{S}_{k,s}(N)
N^{\frac{s}{k}-\frac{1}{2}}+O\Big(N^{\frac{s}{k}-\frac{1}{2}-\frac{1}{k\cdot2^{k-1}}+\varepsilon}\Big).
\end{equation}
Moreover, there exists a function $s_0(k)$ satisfying
\begin{equation*}
 s_0(k)\leqslant\frac{1}{2}k^2\log k+O(k^2\log\log k),
\end{equation*}
and such that (\ref{asymp-dio-error-modify}) holds whenever $s\geqslant s_0(k)$. Note that when $k=3$ we may take
$s=6$, so the formula for $R_{3,6}(n)$ that was already established by Sinnadurai \cite{Sinnadurai-1965} and Hooley \cite{Hooley-1981}, is included in (\ref{asymp-dio-error-modify}). A mild benefit is the explicit exponent in the error term in (\ref{asymp-dio-error-modify}). Connaisseurs will notice that the saving in
(\ref{asymp-dio-error-modify}) over the main term is exactly the saving provided by Weyl's inequality. In such an instance, the method only just fails with one $k$--th power removed from the representation problem. Thus, one would hope to handle the case $s=2^{k-1}+1$ by a further refinement of the basic method. When $k=3$, Watson \cite{Watson-1972} showed the solvability with $s=5$. However, the method of Watson \cite{Watson-1972} only gives a weak estimation for the number of representations. In 1986, Vaughan \cite{Vaughan-1986} enhanced Watson's lower bound to the expected order of magnitude and obtained $R_{3,5}(n)\gg n^{7/6}$. Based on Vaughan's result, it is reasonable to conjecture that every sufficiently large even integer can be represented as a sum of one square of prime and five cubes of primes. However, this conjecture is perhaps out of reach at present. The hitherto best approximation in this direction is due to Li and Zhang \cite{Li-Zhang-2018}, who showed that (\ref{ini-pro}) is solvable for $k=3,s=5$ with $y_1,\dots,y_5$ being primes and $x$ being an almost--prime $P_6$.

For $k=5$, Br\"{u}dern and Kawada \cite{Brudern-Kawada-2011} established the asymptotic formula for $s=17$. To be specific, they showed that 
\begin{equation*}
 R_{5,17}(n)=\frac{\Gamma(\frac{3}{2})\Gamma^{17}(\frac{6}{5})}{\Gamma(\frac{39}{10})}\mathfrak{S}_{5,17}(n)
 n^{\frac{29}{10}}+O(n^{\frac{29}{10}-\delta})
\end{equation*}
holds for some $\delta>0$. In view of the result of Br\"{u}dern and Kawada \cite{Brudern-Kawada-2011}, it is reasonable to conjecture that, for every sufficiently large integer $n$ subject to $n\equiv0\pmod2$, the following equation
\begin{equation*}
 n=p^2+p_1^5+p_2^5+\cdots+p_{17}^5
\end{equation*}
is solvable in prime variables $p,p_1,p_2,\dots,p_{17}$. In this paper, we shall claim that this conjecture is true
and establish the following result.

\begin{theorem}\label{Theorem}
  Every sufficiently large positive integer $n$ subject to $n\equiv0\pmod2$ can be represented as a sum of one square of prime and seventeen fifth powers of primes.
\end{theorem}

\begin{notation}
Throughout this paper, let $p$, with or without subscripts, always denote a prime number; $\varepsilon$  always denotes a sufficiently small positive constant, which may not be the same at different occurrences. As usual, we use $f(x)\ll g(x)$ to denote $f(x)=O(g(x))$.
\end{notation}

\section{Preliminary and outline of the proof of Theorem \ref{Theorem}}
In order to better illustrate Lemma \ref{KW-2.2} and Lemma \ref{A-kth-powers} below, we first introduce some notations and definitions. When $\mathcal{C}\subseteq\mathbb{N}$, we write $\overline{\mathcal{C}}$ for the complement
$\mathbb{N}\setminus\mathcal{C}$ of $\mathcal{C}$ within $\mathbb{N}$. When $a$ and $b$ are non--negative integers, it is convenient to denote by $(\mathcal{C})_a^b$ the set $\mathcal{C}\cap(a,b]$, and by $|\mathcal{C}|_a^b$ the cardinality of $\mathcal{C}\cap(a,b]$. Next, when $\mathcal{C},\mathcal{D}\subseteq\mathbb{N}$, we define
\begin{equation*}
  \mathcal{C}+\mathcal{D}=\{c+d:c\in\mathcal{C}\,\,\,\textrm{and}\,\,\,d\in\mathcal{D}\}.
\end{equation*}
Also, we define $\Upsilon(\mathcal{C},\mathcal{D};N)$ to be the number of solutions of the equation
\begin{equation*}
  c_1-d_1=c_2-d_2,
\end{equation*}
with $c_1,c_2\in(\mathcal{C})_{2N}^{3N}$ and $d_1,d_2\in(\mathcal{D})_{0}^{N}$. It is convenient, when $k$ is a natural number, to describe a subset $\mathcal{Q}$ of $\mathbb{N}$ as being a \textit{high--density subset of the $k$--th powers} when (i) one has $\mathcal{Q}\subseteq\{n^k:n\in\mathbb{N}\}$, and (ii) for each positive number $\varepsilon$, whenever $N$ is a natural number sufficiently large in terms of $\varepsilon$, then
$|\mathcal{Q}|^N_0>N^{1/k-\varepsilon}$. Also, when $\theta>0$, we shall refer to a set $\mathcal{R}\subseteq\mathbb{N}$ as having \textit{complementary density growth exponent smaller than $\theta$} when there exists a positive number $\delta$ with the property that, for all sufficiently large natural numbers $N$, one has $|\overline{\mathcal{R}}|_0^N<N^{\theta-\delta}$.

When $q$ is a natural number and $\mathfrak{a}\in\{0,1,\dots,q-1\}$, we define $\mathcal{P}_\mathfrak{a}=\mathcal{P}_{\mathfrak{a},q}$ by
\begin{equation*}
\mathcal{P}_{\mathfrak{a},q}=\{\mathfrak{a}+mq:m\in\mathbb{Z}\}.
\end{equation*}
Also, we describe a set $\mathcal{L}$ as being a \textit{union of arithmetic progressions modulo $q$} when, for some
subset $\mathfrak{L}$ of $\{0,1,\dots,q-1\}$, one has
\begin{equation*}
\mathcal{L}=\bigcup_{\mathfrak{l}\in\mathfrak{L}}\mathcal{P}_{\mathfrak{l},q}.
\end{equation*}
In such circumstances, given a subset $\mathcal{C}$ of $\mathbb{N}$ and integers $a$ and $b$, it is convenient to write
\begin{equation*}
\langle\mathcal{C}\wedge\mathcal{L}\rangle_a^b =\min_{\mathfrak{l}\in\mathfrak{L}}\big|\mathcal{C}\cap\mathcal{P}_{\mathfrak{l},q}\big|_a^b.
\end{equation*}
Let $\mathcal{L}$ be a union of arithmetic progressions modulo $q$, for some natural number $q$. When $k$ is a natural number, we describe a subset $\mathcal{Q}$ of $\mathbb{N}$ as being a \textit{high--density subset of the $k$--th powers relative to $\mathcal{L}$} when (i) one has $\mathcal{Q}\subseteq\{n^k:n\in\mathbb{N}\}$, and (ii) for each positive number $\varepsilon$, whenever $N$ is a natural number sufficiently large in terms of $\varepsilon$, then $\langle\mathcal{Q}\wedge\mathcal{L}\rangle_0^N\gg_qN^{1/k-\varepsilon}$. In addition, when $\theta>0$, we shall refer to a set $\mathcal{R}\subseteq\mathbb{N}$ as having \textit{$\mathcal{L}$--complementary density growth exponent smaller than $\theta$} when there exists a positive number $\delta$ with the property that, for all sufficiently large natural numbers $N$,
one has $|\overline{\mathcal{R}}\cap\mathcal{L}|_0^N<N^{\theta-\delta}$.

\begin{lemma}\label{KW-2.2}
Let $\mathcal{L}$, $\mathcal{M}$ and $\mathcal{N}$ be unions of arithmetic progressions modulo $q$, for some
natural number $q$, and suppose that $\mathcal{N}\subseteq\mathcal{L}+\mathcal{M}$. Suppose also that $\mathcal{S}$ is a high--density subset of the squares relative to $\mathcal{L}$, and that $\mathcal{A}\subseteq\mathbb{N}$ has
$\mathcal{M}$--complementary density growth exponent smaller than $1$. Then, whenever $\varepsilon>0$ and $N$ is a natural number sufficiently large in terms of $\varepsilon$, one has
\begin{equation*}
   \big|\overline{\mathcal{A}+\mathcal{S}}\cap\mathcal{N}\big|_{2N}^{3N}\ll_q
   N^{-\frac{1}{2}+\varepsilon}\big|\overline{\mathcal{A}}\cap\mathcal{M}\big|_{N}^{3N}.
\end{equation*}
\end{lemma}
\begin{proof}
  See Theorem 2.2 of Kawada and Wooley \cite{Kawada-Wooley-2010}. $\hfill$
\end{proof}

\begin{lemma}\label{KW-2.1}
Suppose that $\mathcal{A},\mathcal{B}\subseteq\mathbb{N}$. In addition, let $\mathcal{L}$, $\mathcal{M}$ and $\mathcal{N}$ be unions of arithmetic progressions modulo $q$, for some natural number $q$, and suppose that $\mathcal{N}\subseteq\mathcal{L}+\mathcal{M}$. Then, for each natural number $N$, one has
\begin{equation*}
  \Big(\langle\mathcal{B}\wedge\mathcal{L}\rangle_0^N\cdot\big|\overline{\mathcal{A}+\mathcal{B}}\cap\mathcal{N}
  \big|_{2N}^{3N}\Big)^2\leqslant q\big|\overline{\mathcal{A}}\cap\mathcal{M}\big|_{N}^{3N}\cdot
  \Upsilon\big(\overline{\mathcal{A}+\mathcal{B}}\cap\mathcal{N},\mathcal{B}\cap\mathcal{L};N\big).
\end{equation*}
\end{lemma}
\begin{proof}
  See Theorem 2.1 of Kawada and Wooley \cite{Kawada-Wooley-2010}. $\hfill$
\end{proof}

\begin{lemma}\label{A-kth-powers}
Let $\mathcal{L}$, $\mathcal{M}$ and $\mathcal{N}$ be unions of arithmetic progressions modulo $q$, for some
natural number $q$, and suppose that $\mathcal{N}\subseteq\mathcal{L}+\mathcal{M}$. Suppose also that, for $k\geqslant3$, $\mathcal{K}$ is a high--density subset of the $k$--th powers relative to $\mathcal{L}$, and that $\mathcal{A}\subseteq\mathbb{N}$ has $\mathcal{M}$--complementary density growth exponent smaller than $\theta$, for some positive number $\theta$. Then, whenever $\varepsilon>0$ and $N$ is a  natural number sufficiently large in terms of $\varepsilon$, without any condition on $\theta$, one has
\begin{equation*}
   \big|\overline{\mathcal{A}+\mathcal{K}}\cap\mathcal{N}\big|_{2N}^{3N}\ll_q
   N^{-\frac{1}{k}+\varepsilon}\big|\overline{\mathcal{A}}\cap\mathcal{M}\big|_{N}^{3N}
   +N^{-\frac{3}{k}+\varepsilon}\Big(\big|\overline{\mathcal{A}}\cap\mathcal{M}\big|_{N}^{3N}\Big)^2.
\end{equation*}
\end{lemma}
\begin{proof}
  Let $N$ be a large natural number, write $\Theta=N^{\frac{1}{k}}$. Observe that since
$\mathcal{K}$ is a high--density subset of the $k$--th powers relative to $\mathcal{L}$, then
$\langle\mathcal{K}\wedge\mathcal{L}\rangle_0^N\gg N^{\frac{1}{k}-\varepsilon}$, and hence it follows from
Lemma \ref{KW-2.1} that
\begin{equation*}
  \Big(\langle\mathcal{K}\wedge\mathcal{L}\rangle_0^N\cdot\big|\overline{\mathcal{A}+\mathcal{K}}\cap\mathcal{N}
  \big|_{2N}^{3N}\Big)^2\ll_q\big|\overline{\mathcal{A}}\cap\mathcal{M}\big|_{N}^{3N}\cdot
  \Upsilon\big(\overline{\mathcal{A}+\mathcal{K}}\cap\mathcal{N},\mathcal{K}\cap\mathcal{L};N\big),
\end{equation*}
which implies that
\begin{equation}\label{A-kth-1}
 \Big(N^{\frac{1}{k}-\varepsilon}\cdot\big|\overline{\mathcal{A}+\mathcal{K}}\cap\mathcal{N}
  \big|_{2N}^{3N}\Big)^2\ll_q\big|\overline{\mathcal{A}}\cap\mathcal{M}\big|_{N}^{3N}\cdot
  \Upsilon\big(\overline{\mathcal{A}+\mathcal{K}}\cap\mathcal{N},\mathcal{K}\cap\mathcal{L};N\big).
\end{equation}
Trivially, the quantity $\Upsilon\big(\overline{\mathcal{A}+\mathcal{K}}\cap\mathcal{N},\mathcal{K}\cap\mathcal{L};N\big)$ is bounded above by the number of solutions of the equation
\begin{equation}\label{A-kth-2}
  n_1-n_2=x^k-y^k,
\end{equation}
with $n_1,n_2\in\big(\overline{\mathcal{A}+\mathcal{K}}\cap\mathcal{N}\big)_{2N}^{3N}$ and $1\leqslant x,y\leqslant\Theta$. Write $\mathcal{Z}(N)$ for $\big(\overline{\mathcal{A}+\mathcal{K}}\cap\mathcal{N}\big)_{2N}^{3N}$.
Also, define the exponential sums as follows
\begin{equation*}
  f(\alpha)=\sum_{1\leqslant x\leqslant \Theta}e(\alpha x^k) \qquad \textrm{and}\qquad K(\alpha)=\sum_{n\in\mathcal{Z}(N)}e(n\alpha).
\end{equation*}
Then, on considering the underlying Diophantine equation, it follows from (\ref{A-kth-2}) that
\begin{equation}\label{A-kth-3}
  \Upsilon\big(\overline{\mathcal{A}+\mathcal{K}}\cap\mathcal{N},\mathcal{K}\cap\mathcal{L};N\big)
  \leqslant\int_0^1\big|f^2(\alpha)K^2(\alpha)\big|\mathrm{d}\alpha.
\end{equation}
According to Lemma 6.1 of Kawada and Wooley \cite{Kawada-Wooley-2010} with $j=1$, we get
\begin{equation}\label{A-kth-4}
  \int_0^1\big|f^2(\alpha)K^2(\alpha)\big|\mathrm{d}\alpha\ll N^{\frac{1}{k}}
  \big|\overline{\mathcal{A}+\mathcal{K}}\cap\mathcal{N}\big|_{2N}^{3N}+N^{\frac{1}{2k}+\varepsilon}
  \Big(\big|\overline{\mathcal{A}+\mathcal{K}}\cap\mathcal{N}\big|_{2N}^{3N}\Big)^{\frac{3}{2}}.
\end{equation}
Combining (\ref{A-kth-1}), (\ref{A-kth-3}) and (\ref{A-kth-4}), we conclude that
\begin{equation*}
\big|\overline{\mathcal{A}+\mathcal{K}}\cap\mathcal{N}\big|_{2N}^{3N}
\ll_qN^{-\frac{1}{k}+\varepsilon}\big|\overline{\mathcal{A}}\cap\mathcal{M}\big|_{N}^{3N}
+N^{-\frac{3}{2k}+\varepsilon}\big|\overline{\mathcal{A}}\cap\mathcal{M}\big|_{N}^{3N}
\cdot\Big(\big|\overline{\mathcal{A}+\mathcal{K}}\cap\mathcal{N}\big|_{2N}^{3N}\Big)^{\frac{1}{2}},
\end{equation*}
which implies that
\begin{equation*}
\big|\overline{\mathcal{A}+\mathcal{K}}\cap\mathcal{N}\big|_{2N}^{3N}
\ll_qN^{-\frac{1}{k}+\varepsilon}\big|\overline{\mathcal{A}}\cap\mathcal{M}\big|_{N}^{3N}
+N^{-\frac{3}{k}+\varepsilon}\Big(\big|\overline{\mathcal{A}}\cap\mathcal{M}\big|_{N}^{3N}\Big)^2.
\end{equation*}
This completes the proof of Lemma \ref{A-kth-powers}. $\hfill$
\end{proof}

\begin{lemma}\label{exceptional-start}
Let $E_{5,s}(N)$ denote the number of positive integers $n$ up to $N$ subject to $n\equiv s\pmod 2$
for which $n=\sum_{k=1}^{s}p_k^5$ is not solvable in prime variables $p_1,\dots,p_s$ . Then, for any $\varepsilon>0$, there holds
\begin{equation*}
  E_{5,15}(N)\ll N^{1-\frac{1}{10}-\frac{27}{3200}+\varepsilon}.
\end{equation*}
\end{lemma}
\begin{proof}
See Theorem 1.1 of Chen \cite{Chen-Gongrui-2022}.
\end{proof}

\section{Proof of Theorem \ref{Theorem}}

Let $E(N)$ denote the set of positive integers $n$ subject to $n\equiv0\pmod2$ up to $N$, which can not be represented as
\begin{equation}\label{conjecture-11}
   n=p^2+p_1^5+p_2^5+\cdots+p_{17}^5.
\end{equation}
The remaining part of this section is devoted to establishing Theorem \ref{Theorem} by using Lemma \ref{KW-2.2}, Lemma \ref{A-kth-powers} and Lemma \ref{exceptional-start}.

\noindent
\textit{Proof of Theorem \ref{Theorem}} \quad Define
\begin{equation*}
   \mathcal{A}_1=\Big\{ p_1^5+p_2^5+\cdots+p_{15}^5: \,\,\textrm{$p_j$'s are primes, $j=1,2,\dots,15$}\Big\},
\end{equation*}
\begin{equation*}
\mathcal{A}_2=\Big\{p^2+p_1^5+p_2^5+\cdots+p_{15}^5: \,\,\textrm{$p$ and $p_j$'s are primes, $j=1,2,\dots,15$}\Big\},
\end{equation*}
\begin{equation*}
\mathcal{A}_3=\Big\{p^2+p_1^5+p_2^5+\cdots+p_{16}^5: \,\,\textrm{$p$ and $p_j$'s are primes, $j=1,2,\dots,16$}\Big\},
\end{equation*}
\begin{equation*}
   \mathcal{M}=\mathcal{L}=\mathcal{N}_2=\big\{n\in\mathbb{N}^+:\,\, n\equiv1\!\!\!\!\!\pmod 2\big\},
\end{equation*}
\begin{equation*}
   \mathcal{N}_1=\mathcal{N}_3=\big\{n\in\mathbb{N}: \,\, n\equiv0\!\!\!\!\!\pmod 2\big\},\qquad
   \mathcal{K}_j=\big\{p^{\,j}:\, p \,\,\textrm{is a prime}\big\},
\end{equation*}
\begin{equation*}
   \mathscr{E}_1=\Big\{n\in\mathbb{N}:\,\, n\in\mathcal{M},\,\,
   n\not=p_1^5+p_2^5+\cdots+p_{15}^5,\,\,\textrm{$p_j$'s are primes} \Big\},
\end{equation*}
\begin{equation*}
   \mathscr{E}_2=\Big\{n\in\mathbb{N}:\,\, n\in\mathcal{N}_1,\,\, n\not=p^2+p_1^5+p_2^5+\cdots+p_{15}^5,\,\,
   \textrm{$p$ and $p_j$'s are primes} \Big\},
\end{equation*}
\begin{equation*}
   \mathscr{E}_3=\Big\{n\in\mathbb{N}:\,\, n\in\mathcal{N}_2,\,\,n\not=p^2+p_1^5+p_2^5+\cdots+p_{16}^5,\,\,
   \textrm{$p$ and $p_j$'s are primes} \Big\},
\end{equation*}
\begin{equation*}
   \mathscr{E}=\Big\{n\in\mathbb{N}:\,\, n\in\mathcal{N}_3,\,n\not=p^2+p_1^5+p_2^5+\cdots+p_{17}^5,\,\,
   \textrm{$p$ and $p_j$'s are primes}\Big \}.
\end{equation*}
Thus, one has $E(N)=\big|\mathscr{E}\big|_0^N$. Moreover, we denote $\big|\mathscr{E}_1\big|_0^N,\big|\mathscr{E}_2\big|_0^N$  and $\big|\mathscr{E}_3\big|_0^N$ by $E_1(N),E_2(N)$ and $E_3(N)$, respectively. Trivially, $\mathcal{L},\mathcal{M},\mathcal{N}_1,\mathcal{N}_2,\mathcal{N}_3$ are a union of arithmetic progression modulo $2$ subject to
\begin{equation*}
  \begin{cases}
   \mathcal{N}_1\subseteq\mathcal{L}+\mathcal{M}, \\
   \mathcal{N}_2\subseteq\mathcal{L}+\mathcal{N}_1, \\
   \mathcal{N}_3\subseteq\mathcal{L}+\mathcal{N}_2,
  \end{cases}
\end{equation*}
Moreover, it follows from the Prime Number Theorem in arithmetic progression that there hold
\begin{equation*}
  \langle\mathcal{K}_2\wedge\mathcal{L}\rangle_0^N\gg N^{\frac{1}{2}}(\log N)^{-1},\qquad
  \langle\mathcal{K}_5\wedge\mathcal{L}\rangle_0^N\gg N^{\frac{1}{5}}(\log N)^{-1}.
\end{equation*}
Therefore, $\mathcal{K}_2$ and $\mathcal{K}_5$ are two high--density subsets of the squares and fifth powers, respectively, relative to $\mathcal{L}$. By Lemma \ref{exceptional-start}, it is easy to see that
\begin{equation*}
  \big|\overline{\mathcal{A}_1}\cap\mathcal{M}\big|_0^N=\big|\mathscr{E}_1\big|_0^N=E_1(N)\ll N^{1-(\frac{1}{10}+\frac{27}{3200})+\varepsilon}.
\end{equation*}
Thus, $\mathcal{A}_1$ has $\mathcal{M}$ complementary density growth exponent smaller than $1$.
From Lemma \ref{KW-2.2}, we know that
\begin{align*}
             \big|\mathscr{E}_2\big|_{2N}^{3N}=\big|\overline{\mathcal{A}_1+\mathcal{K}_2}
             \cap\mathcal{N}_1\big|_{2N}^{3N}
  \ll & \,\, N^{-\frac{1}{2}+\varepsilon}\big|\overline{\mathcal{A}_1}\cap\mathcal{M}\big|_N^{3N}
                 \nonumber \\
  \ll & \,\, N^{-\frac{1}{2}+\varepsilon}\cdot E_1(3N)
             \ll N^{\frac{1}{2}-(\frac{1}{10}+\frac{27}{3200})+\varepsilon}.
\end{align*}
Let the integers $N_j$ for $j\geqslant0$ by means of the iterative formula
\begin{equation}\label{N_0-ita}
  N_0=\bigg\lceil\frac{1}{2}N\bigg\rceil,\qquad N_{j+1}=\bigg\lceil\frac{2}{3}N_j\bigg\rceil, \quad (j\geqslant0),
\end{equation}
where $\lceil N\rceil$ denotes the least integer not smaller than $N$. Moreover, we define $J$ to be the least positive integer with the property that $N_j\leqslant10$, then $J\ll\log N$. Therefore, there holds
\begin{equation}\label{E_2-upper-1}
  E_2(N)\leqslant10+\sum_{j=1}^J\big|\mathscr{E}_2\big|_{2N_j}^{3N_j}
         \ll N^{\frac{1}{2}-(\frac{1}{10}+\frac{27}{3200})+\varepsilon}.
\end{equation}
 By (\ref{E_2-upper-1}), we know that
\begin{equation*}
  \big|\overline{\mathcal{A}_2}\cap\mathcal{N}_1\big|_0^N=\big|\mathscr{E}_2\big|_0^N=E_2(N)\ll N^{\frac{1}{2}-(\frac{1}{10}+\frac{27}{3200})+\varepsilon}.
\end{equation*}
Hence, $\mathcal{A}_2$ has $\mathcal{N}_1$ complementary density growth exponent smaller than $\frac{1}{2}$. From
Lemma \ref{A-kth-powers} with $k=5$, we obtain
\begin{align*}
             \big|\mathscr{E}_3\big|_{2N}^{3N}=\big|\overline{\mathcal{A}_2+\mathcal{K}_5}
             \cap\mathcal{N}_2\big|_{2N}^{3N}
\ll & \,\, N^{-\frac{1}{5}+\varepsilon}\big|\overline{\mathcal{A}_2}\cap\mathcal{N}_1\big|_N^{3N}
          +N^{-\frac{3}{5}+\varepsilon}\Big(\big|\overline{\mathcal{A}_2}\cap\mathcal{N}_1\big|_N^{3N}\Big)^2
               \nonumber \\
\ll & \,\, N^{-\frac{1}{5}+\varepsilon}\cdot E_2(3N)+N^{-\frac{3}{5}+\varepsilon}\big(E_2(3N)\big)^2
                \nonumber \\
\ll & \,\, N^{-\frac{1}{5}+\varepsilon}\cdot N^{\frac{1}{2}-(\frac{1}{10}+\frac{27}{3200})+\varepsilon}
           +N^{-\frac{3}{5}+\varepsilon}\big(N^{\frac{1}{2}-(\frac{1}{10}+\frac{27}{3200})+\varepsilon}\big)^2
                \nonumber \\
\ll & \,\, N^{\frac{1}{5}-\frac{27}{3200}+\varepsilon}.
\end{align*}
By the same notation of (\ref{N_0-ita}), we derive that
\begin{equation}\label{E_3-upper-1}
  E_3(N)\leqslant10+\sum_{j=1}^J\big|\mathscr{E}_3\big|_{2N_j}^{3N_j}\ll N^{\frac{1}{5}-\frac{27}{3200}+\varepsilon}.
\end{equation}
It follows from (\ref{E_3-upper-1}) that
\begin{equation*}
  \big|\overline{\mathcal{A}_3}\cap\mathcal{N}_2\big|_0^N=\big|\mathscr{E}_3\big|_0^N=E_3(N)\ll N^{\frac{1}{5}-\frac{27}{3200}+\varepsilon}.
\end{equation*}
Consequently, $\mathcal{A}_3$ has $\mathcal{N}_2$ complementary density growth exponent smaller than $\frac{1}{5}$. By Lemma \ref{A-kth-powers} with $k=5$ again, we derive that
\begin{align*}
  \big|\mathscr{E}\big|_{2N}^{3N}=\big|\overline{\mathcal{A}_3+\mathcal{K}_5}\cap\mathcal{N}_3\big|_{2N}^{3N}
  \ll & \,\, N^{-\frac{1}{5}+\varepsilon}\big|\overline{\mathcal{A}_3}\cap\mathcal{N}_2\big|_N^{3N}
            +N^{-\frac{3}{5}+\varepsilon}\Big(\big|\overline{\mathcal{A}_3}\cap\mathcal{N}_2\big|_N^{3N}\Big)^2
                 \nonumber \\
  \ll & \,\, N^{-\frac{1}{5}+\varepsilon}\cdot E_3(3N)+N^{-\frac{3}{5}+\varepsilon}\big(E_3(3N)\big)^2
                \nonumber \\
  \ll & \,\, N^{-\frac{1}{5}+\varepsilon}\cdot N^{\frac{1}{5}-\frac{27}{3200}+\varepsilon}+
             N^{-\frac{3}{5}+\varepsilon}\cdot\big(N^{\frac{1}{5}-\frac{27}{3200}+\varepsilon}\big)^2
                \nonumber \\
  \ll & \,\, N^{-\frac{27}{3200}+\varepsilon}.
\end{align*}
At last, with the same notation of (\ref{N_0-ita}) again, we deduce that
\begin{equation*}
   E(N)\leqslant10+\sum_{j=1}^J\big|\mathscr{E}\big|_{2N_j}^{3N_j}\ll 1,
\end{equation*}
which implies that the exceptional set of the representation of $n$ satisfying necessary congruent condition,
i.e., $n\equiv0\pmod2$, as the sum of one square of prime and seventeen fifth powers of primes is bound.
This completes the proof of Theorem \ref{Theorem}.

\vskip 6mm

\section*{Acknowledgement}

 The authors would like to express the most sincere gratitude to the referee for his/her patience in refereeing this paper. This work is supported by the Natural Science Foundation of Beijing Municipal (Grant No. 1242003),
 and the National Natural Science Foundation of China (Grant Nos. 12001047, 11901566, 11971476, 12071238).


\begin{thebibliography}{99}

\bibitem{Brudern-Kawada-2011}J. Br\"{u}dern, K. Kawada, \textit{The asymptotic formula in Waring's problem for
     one square and seventeen fifth powers}, Monatsh. Math., \textbf{162} (2011), no. 4, 385--407.

\bibitem{Chen-Gongrui-2022}G. Chen, \textit{On exceptional sets in the Waring--Goldbach problem for fifth powers},
                        Ramanujan J., \textbf{62} (2023), no. 1, 329--346.

\bibitem{Hardy-Littlewood-1920}G. H. Hardy, J. E. Littlewood, \textit{Some problems of `Partitio Numerorum'
                (I): a new solution to Waring's problem}, G\"{o}ttingen Nachrichten, (1920), no. 1--2, 33--54.

\bibitem{Hardy-Littlewood-1925}G. H. Hardy, J. E. Littlewood, \textit{Some problems of `partitio numerorum'
                (VI): further researches in Waring's problem}, Math. Z., \textbf{23} (1925), no. 1, 1--37.

\bibitem{Hooley-1981}C. Hooley, \textit{On a new approach to various problems of Waring's type}, in: Recent
                    progress in analytic number theory, Vol. 1, Academic Press, London, 1981, pp. 127--191.

\bibitem{Kawada-Wooley-2010}K. Kawada, T. D. Wooley, \textit{Relations between exceptional sets for
                        additive problems}, J. Lond. Math. Soc. (2), \textbf{82} (2010), no. 2, 437--458.

\bibitem{Li-Zhang-2018}J. Li, M. Zhang, \textit{On the Waring--Goldbach problem for one square and five cubes},
                         Int. J. Number Theory, \textbf{14} (2018), no. 9, 2425--2440.

\bibitem{Sinnadurai-1965}J. St. C. L. Sinnadurai, \textit{Representation of integers as sums of six cubes
                       and one square}, Quart. J. Math. Oxford Ser. (2), \textbf{16} (1965), 289--296.

\bibitem{Stanley-1930}G. K. Stanley, \textit{The representation of a number as the sum of one square and
             a number of $k$th powers}, Proc. London Math. Soc. (2), \textbf{31} (1930), no. 1, 512--553.

\bibitem{Vaughan-1986}R. C. Vaughan, \textit{On Waring's problem: one square and five cubes},
                       Quart. J. Math. Oxford Ser. (2), \textbf{37} (1986), no. 1, 117--127.

\bibitem{Watson-1972}G. L. Watson, \textit{On sums of a square and five cubes},
                                  J. London Math. Soc. (2), \textbf{5} (1972), 215--218.












\end{thebibliography}
\end{document}